\documentclass[11pt]{amsart}
\usepackage{geometry}                % See geometry.pdf to learn the layout options. There are lots.
\usepackage{graphicx}
\usepackage{amssymb}
\usepackage{epstopdf}
\usepackage{hyperref}

\newtheorem{thm}{Theorem}

\newtheorem{cor}{Corollary}
\newtheorem{prop}{Proposition}
\newtheorem{conj}{Conjecture}

\newcommand{\Q}{\mathbb{Q}}
\newcommand{\R}{\mathbb{R}}
\newcommand{\Z}{\mathbb{Z}}

\title{On the distribution of rational squares}
\author{Michael Weiss}
\begin{document}
\maketitle
\section{Introduction}

This paper explores the location of rational squares in the denominator-first lexicographic ordering of an open interval of the form $(a,b) \subset \R^+$.  In the usual ordering on $\R^+$, the rational squares are dense, in the sense that given any two positive $a,b \in \R^+$, there exists $q \in \Q$ such that $a<q^2<b$.  This follows easily from the fact that the rationals are dense in $\R$, and hence we may choose $q \in (\sqrt{a},\sqrt{b})$.  But in the denominator-first lexicographic ordering, the problem of locating rational squares is non-trivial.  By ``denominator-first lexicographic ordering'' we refer to the ordering of rationals in which we first  list all of the rationals in $(a,b)$ of the form $k/1$ (if indeed there are any), in increasing order of their numerator; then we list all of the rationals of the form $k/2$ (if any), again in increasing order of numerator, and omitting any that have already been listed in lower terms; then all of the rationals of the form $k/3$; and so on.  This enumeration, well-known from Cantor's proof that the rationals in $(0,1)$ are countable, produces a well-ordering of $(a,b) \cap \Q$, and we may ask:  Where, in this ordering, do we encounter rational squares -- and in particular, where is the \em first \em rational square?

In this paper, I consider the case where the interval $(a,b)$ is of the form $(a,a+1)$ for $a \in \Z^+ $ (although several of the results proved below hold equally well for $a \in \R^+$). That is, we seek natural numbers $t,s  \in \Z^+$ such that $a < (t/s)^2 < a+1$ or, equivalently, such that $s^2 a < t^2 < s^2 (a+1)$.  I ask the following questions:

\begin{enumerate}

\item  For a given $a \in \Z^+$, what is the least positive integer $s$ such that an integer square lies between $s^2 a$ and $s^2 (a+1)$?  Let $\sigma (a)$ denote the answer to this question.

\item  For given $a \in \Z^+$ and $s \in \Z^+$, how many positive integers $t$ satisfy  $s^2 a < t^2 < s^2 (a+1)$?  Let $\tau_s (a)$ denote the answer to this question.

\end{enumerate}

It will occasionally be useful to refer to the set $T_s(a)$ of all positive integers $t$ such that $s^2 a < t^2 < s^2 (a+1)$; with this definition, $\tau_s(a) = | T_s(a) |$, and $\sigma(a)$ is the smallest value of $s$ such that $T_s(a) \neq \varnothing$.  Equivalently, we may describe $\sigma(a)$ as the smallest denominator among all rational numbers in the open interval $(\sqrt{a}, \sqrt{a+1})$; as the least $s \in \Z^+$ such that the open interval $(s \sqrt{a} , s \sqrt{a+1})$ contains an integer; or as the least $s$ such that
$ s\sqrt{a+1} > 1 + \lfloor s  \sqrt{a}  \rfloor$.  The goal of this paper is to explore the behavior of the sequence $(\tau_s(a))_{s \in \Z^+}$ and of the function  $\sigma (a)$.

Both $\sigma (a)$ and $\tau_s(a)$ can be computed by a relatively simple algorithm.  For example, with $a=8$, we compute $\sigma(8)$ and the first few terms of $(\tau_s(8))_{s \in \Z^+}$ as follows:
\begin{itemize}
\item $2^2 \cdot 8 = 32$ and $2^2 \cdot 9 = 36$, but there are no integer squares strictly between 32 and 36, so $\tau_2(8)=0$ and $\sigma(8)>2$.
\item $3^2 \cdot 8 = 72$ and $3^2\cdot 9 = 81$, but there are no integer squares strictly between 72 and 81, so $\tau_3(8)=0$ and $\sigma(8)>3$.
\item $4^2 \cdot 8 = 128$ and $4^2\cdot 9 = 144$, but again there are no integer squares strictly between 128 and 144, so $\tau_4(8)=0$ and $\sigma(8)>4$.
\item $5^2 \cdot 8 = 200$ and $5^2\cdot 9 = 225$, but again there are no integer squares strictly between 200 and 225, so $\tau_5(8)=0$ and $\sigma(8)>5$.
\item $6^2 \cdot 8 = 288$ and $6^2\cdot 9 = 324$, but this time we find that $17^2 = 289$ lies in the desired range, and is the only such square; so $\tau_6(8)=1$, $\sigma(8)=6$, and $289/36$ is the first rational square in the interval $(8,9)$ in the denominator-first lexicographic ordering.

\end{itemize}

It is known that the first rational (in the sense of this ordering) in any interval $(x,y)$ may be computed from the continued fraction representations of $x$ and $y$.  Specifically, we have the following:

\begin{prop} \cite[p.~654]{knuth1997}
If $x$ and $y$ are irrationals with $0<x<y$ and with regular continued fractions given by

$$x = [a_0; a_1, a_2, \dots, a_{k-1}, a_k, a_{k + 1},\dots]$$
and 
$$y = [a_0; a_1, a_2, \dots, a_{k-1}, b_k, b_{k + 1}, \dots]$$
then the finite continued fraction $z=[a_0; a_1, a_2, \dots, a_{k-1}, \min(a_k,b_k)+1]$ is the first rational in the interval $(x,y)$.
\end{prop}

To illustrate how this may be used, suppose we wish to compute $\sigma(991)$.  This is equivalent to seeking the first rational number in the interval $(\sqrt{991}, \sqrt{992})$.  We have the continued fraction representations
$$\sqrt{991} = [31 ;2, 12, 10, 2, \dots]$$
$$\sqrt{992} = [31; 2, 62, 2, 62, \dots]$$
and thus the rational we seek is $[31; 2, 13] = \frac{850}{27}$, so $\sigma(991)=27$. We will see below that for any $a$, $\tau_{\sigma(a)}(a)=1$, so in fact
$T_{27}(991) = \{ 850 \} $.

If $a \in \Z^+$  and neither $a$ nor $a+1$ are square, then $\sigma(a)$ is the $a^{\text{th}}$ term in sequence A183162 in the Online Encyclopedia of Integer Sequences \cite{OEIS}, which may be defined as the least $s$ such that the half-open interval $(s \sqrt{a} , s \sqrt{a+1}]$ contains an integer.  Note that the definition of $\sigma(a)$ considered in this paper requires that we find $s$ in the \em{open }\em interval $(s \sqrt{a} , s \sqrt{a+1})$, which excludes the trivial solution $s=1$ that would otherwise be admitted when one of $a$ or $a+1$ is already a square; other natural variations of this problem are to seek integers in $[s \sqrt{a} , s \sqrt{a+1})$ or $[s \sqrt{a} , s \sqrt{a+1}]$.  

We note also that if we define $F(a,s) = \lfloor   s \sqrt{a+1} -   \lfloor s \sqrt{a} \rfloor \rfloor$, then $ \tau_s(a) $ may be computed directly from 
\[ \tau_s(a)  =   \begin{cases} 
      F(a,s), & a\neq n^2-1 \\
      F(a,s)-1, & a = n^2-1 
   \end{cases}
\]

Values for $\sigma(a)$ for $a \leq 500$ are graphed in Figure~\ref{fig:sigma} below.   The graph is surprisingly complex.  We note that values of $\sigma(a)$ fluctuate in a seemingly chaotic pattern, with intermittent spikes located at values of $a$ that are integer squares; that the peak values at those spikes seem to track along some smooth upper-bounding curve; and that between those spikes the values of the functions fluctuate above some trough-shaped lower-bound.  In the sections below, I describe the behavior at the spikes and find sharp upper and lower bounds for $\sigma(a)$ between the spikes.

\begin{figure}[htbp]
  \centering
  \includegraphics[scale=.5]{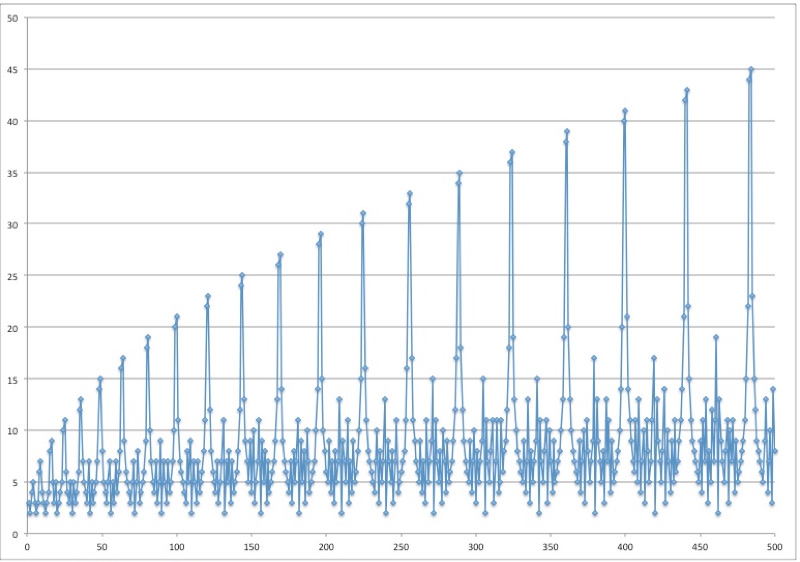}
  \caption{$\sigma(a), 1 \leq a \leq 500$}
  \label{fig:sigma}
\end{figure}

\section{Fundamental properties of $\tau_s(a)$ and $\sigma(a)$}

It is immediate from the definition that $\sigma(a) \geq 2$ for all $a \in \Z^+$.  We begin with the observation that this minimum is obtained for $a=n^2 + n$:

\begin{prop}  For $n \in \Z^+$, $\sigma (n^2+n) = 2$ and $T_2(n^2+n) = \{ 2n+1 \}$.  \end{prop}
\begin{proof}  We note simply that $2^2(n^2+n) < (2n+1)^2 < 2^2(n^2+n+1) < (2n + 2)^2$, from which the claim follows.  \end{proof}

The following Theorem establishes some fundamental properties of the the sequence $(\tau_s(a))_{s \in \Z^+}$.

\begin{thm}  Let $k \in \Z^+$ and $a \in \Z^+$.   Then:
\begin{enumerate}
\item If $s > k(\sqrt{a} +  \sqrt{a+1})$, then $\tau_s(a) \geq k$, and hence $\tau_s(a) \to \infty$ as $s \to \infty$.
\item Successive terms in the sequence $(\tau_s(a))_{s \in \Z^+}$ differ by either $0,1$ or $-1$.
\item The first nonzero term of $\tau_s(a)$ is 1; that is, if $s$ is minimal such that $(s^2 a, s^2 (a+1))$ contains a square integer, that square integer is unique.

\end{enumerate}
\end{thm}

\begin{proof}
For (1), we note that if $s > k(\sqrt{a} + \sqrt{a+1})$ then $s \sqrt{a+1}  - s \sqrt{a} > k$, and hence the open interval $(s \sqrt{a}, s\sqrt{a+1})$ contains at least $k$ positive integers.

For (2), write $n = \lfloor \sqrt{a} \rfloor$.  Suppose that 
$$ \{ t, t + 1, t + 2, \dots, \ t + k - 1, t + k \} \subset T_s(a)$$.
Then
$$ \{ t-n, t + 1-n, t + 2 - n, \dots, \ t + k - 1-n \} \subset T_{s-1}(a)$$
because $s \sqrt{a} < t$ and $n \leq \sqrt{a}$ imply that $(s-1) \sqrt{a} < t - n$; and $t + k < s\sqrt{a+1}$ and $\sqrt{a+1} \leq n + 1$ imply that $t + (k-1) - n < (s-1) \sqrt{a+1}$.  By similar arguments, if
$$ \{ t, t + 1, t + 2, \dots, \ t + k - 1, t + k \} \subset T_s(a)$$.
then also
$$ \{ t + 1 + n, t + 2 + n,  \dots, \ t + k + n \} \subset T_{s+1}(a)$$

Property (3) is an immediate consequence of (2).
\end{proof}

We turn now to the function $\sigma(a)$.  
The following theorem establishes both an upper and lower bound for $\sigma(a)$.

\begin{thm} 

Fix $a \in \Z^+$,  and write $n = \lfloor \sqrt{a} \rfloor$ (so that $a = n^2 + b$ for some $0 \leq b \leq 2n$) and  $m = n+1$ (so that  $a = m^2 - c$ for some $1 \leq c \leq 2m-1$).  With this notation, let $\sigma_l(a) = \frac {n + \sqrt{a+1}}{b+1} $ and $\sigma_r(a) = \frac {m + \sqrt{a}}{c} $.  Then:

\begin{enumerate}
\item For $k \in \Z^+$, if $s \leq  \max \{k \cdot \sigma_l(a), k \cdot \sigma_r(a) \}$ then $\tau_s(a) < k$.
\item Define $\sigma_k(a) = \lfloor \max \{ k \cdot \sigma_l(a) , k \cdot \sigma_r(a) \} \rfloor + 1$.  Then for all $a$, $\tau_s(a) \geq k \implies s \geq \sigma_k(a)$.
\item $\sigma_1(a) \leq \sigma (a) \leq \lceil  \sqrt{a+1}+\sqrt{x}  \rceil$

\end{enumerate}
\end{thm}

\begin{proof}

\begin{enumerate}

\item If $s \leq k \cdot \frac{n + \sqrt{a+1}}{b+1}$, then $(sn+k)^2 \geq s^2 (a+1)$, and therefore the interval $(s \sqrt{a}, s \sqrt{a+1})$ contains at most the $k-1$ integers $sn+1, sn+2, \dots sn+k-1$.  Likewise if $s \leq k \cdot \frac{ m + \sqrt{a} }{c}$, then $(sm-k)^2 \leq s^2a$, and therefore the interval $(s \sqrt{a}, s \sqrt{a+1})$ contains at most the $k-1$ integers $sm-1, sm-2, \dots sm-(k-1)$.  Now if $s \leq k \cdot \max \{ \sigma_l(a), \sigma_r(a) \}$ then at least one of these two conditions holds, whence we have $\tau_s(a) < k$.

\item is immediate from (1).

\item  The lower bound is  (2) with $k=1$, and the upper bound follows from Theorem 1, property (1) with $k=1$.  
\end{enumerate}
\end{proof}

\begin{figure}[htbp]
  \centering
  \includegraphics[scale=.25]{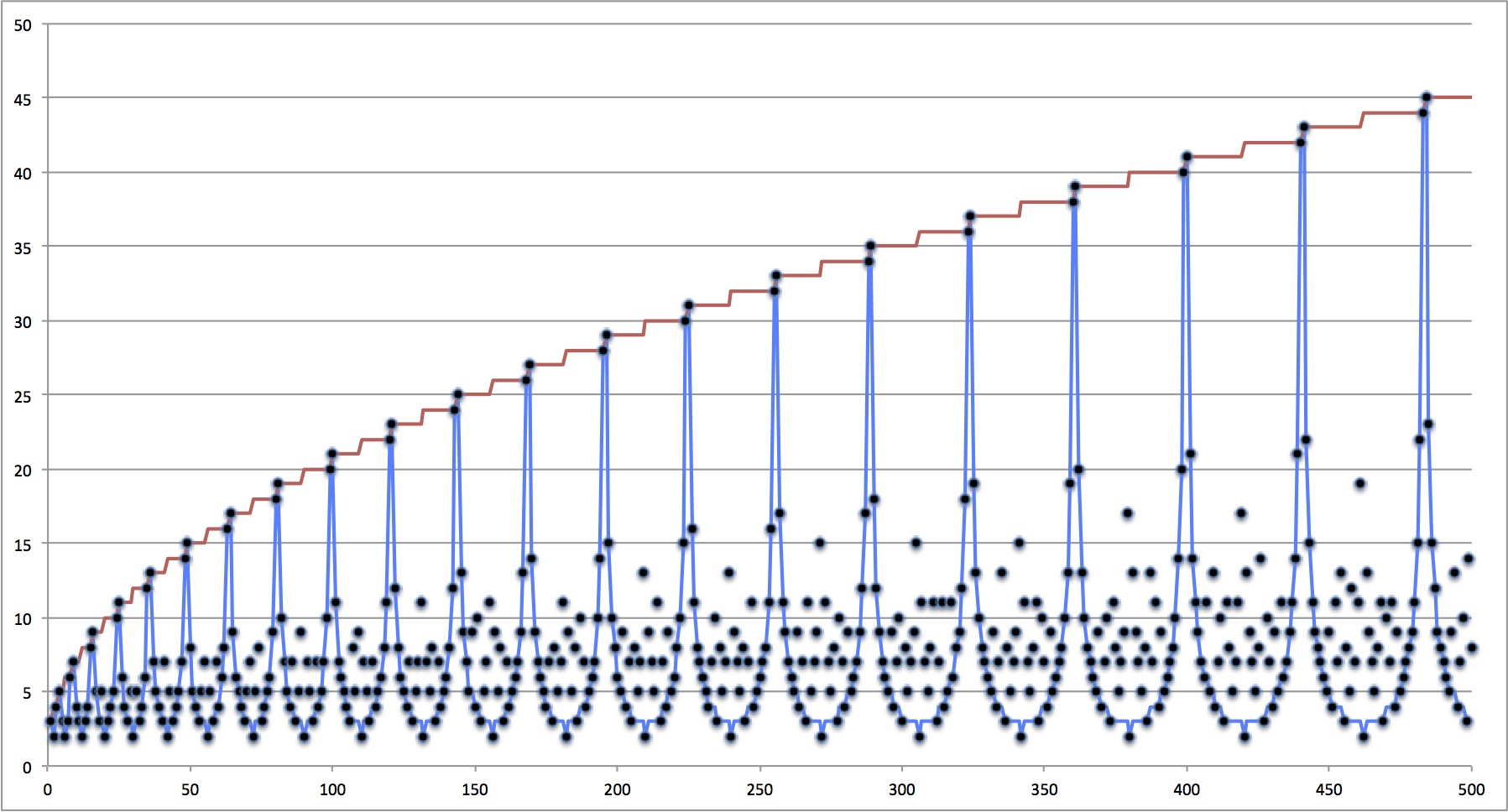}
  \caption{$\sigma(a), 1 \leq a \leq 500, a \in \Z^+$, with upper and lower bounds}
  \label{fig:bounds}
\end{figure}

Figure~\ref{fig:bounds} shows values of $\sigma(a), a \in \Z^+, a \leq 500$ (shown as disconnected points) together with the lower bound (in blue) and upper bound (in red).  We note that 62.8\% of the data points in this interval lie on the lower bound $\sigma_1(a)$, with the rest fluctuating in a seemingly chaotic fashion above $\sigma_1(a)$.  We also see from Figure~\ref{fig:bounds} that the upper and lower bounds coincide for $a=n^2$ and for $a=n^2-1$; in fact, we have the following corollary:

\begin{cor}
\begin{enumerate}
\item For $n \in \Z^+$, $\sigma(n^2) = 2n+1$ and $\sigma(n^2-1) = 2n$.
\item For $n \in \Z^+$, $T_{2n+1}(n^2) = \{ 2n^2 + n + 1 \}$  and $T_{2n}(n^2-1) = \{2n^2 - 1 \}$.
\end {enumerate}
\end{cor}

\begin{proof}
\begin{enumerate}
\item By Theorem 2, we have
$\sigma(n^2) \geq \lfloor \sigma_l(n^2) \rfloor + 1 = \lfloor n + \sqrt{n^2+1} \rfloor +1 = 2n+1$,
but also
$\sigma(n^2) \leq  \lceil  \sqrt{n^2+1}+\sqrt{n^2}  \rceil = 2n + 1$.  Likewise,
$\sigma(n^2-1) \geq \lfloor \sigma_r(n^2-1) \rfloor + 1 = \lfloor n + \sqrt{n^2-1} \rfloor +1 = 2n$
but also
$\sigma(n^2-1) \leq  \lceil  \sqrt{n^2}+\sqrt{n^2-1}  \rceil = 2n$.

\item We note that the inequality $ (2n+1)^2 n^2 < (2n^2 + n + 1)^2 < (2n+1)^2(n^2 + 1)  $ (which may be verified directly) shows that $2n^2 + n + 1 \in T_{2n+1}(n^2)$.  By (1) and Thm. 1 (3), $2n+1$ is the only element in $T_{2n+1}(n^2)$.  The proof that $T_{2n}(n^2-1) = \{2n^2 - 1 \}$ is similar.
\end{enumerate}
\end{proof}

We conclude this section with two visual representations of the structure of these sets.  Figure~\ref{fig:T-sets-heat-map} shows a grid representing pairs $(a,s)$,  $8 \leq a \leq 256$ and $2 \leq s \leq 100$, with cells shaded according to the size of $\tau_s(a)$.  Figure~\ref{fig:T-sets-heat-map} thus constitutes a visual answer to the question ``Where are the rational squares?''  Figure~\ref{fig:delta-tau-heat-map} shows a similar grid, with cells colored 
black if 
$\tau_{s}(a) - \tau_{s-1}(a) = 1$, 
red if
$\tau_{s}(a) - \tau_{s-1}(a) = -1$, 
and white if 
$\tau_{s}(a) - \tau_{s-1}(a) = 0$.  (By Theorem 1, Property (2), these are the only possibilities.)  Thus the vertical columns of Figure~\ref{fig:T-sets-heat-map} show how the sequences $\tau_s(a)$ change with $s$. 

\begin{figure}[htbp]
  \centering
  \includegraphics[scale=.7]{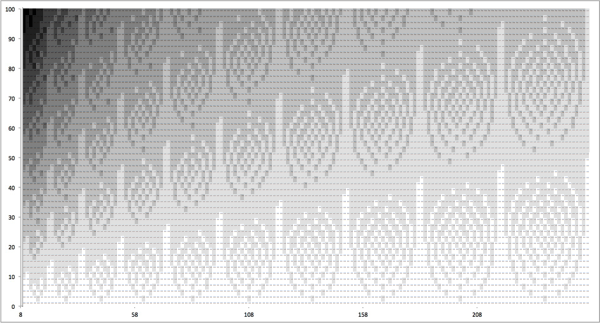}
  \caption{Density of $\tau_s(a), 8 \leq a \leq 256$}
  \label{fig:T-sets-heat-map}
\end{figure}

\begin{figure}[htbp]
  \centering
  \includegraphics[scale=.7]{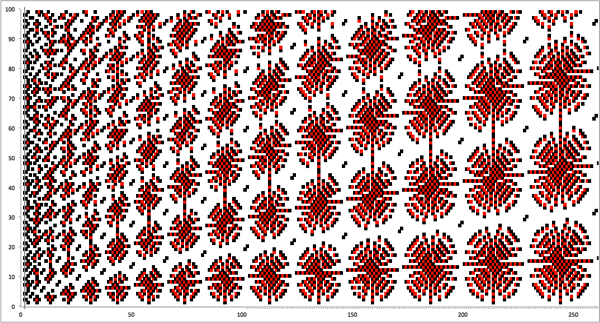}
  \caption{Changes in $\tau_s(a), 1 \leq a \leq 256$}
  \label{fig:delta-tau-heat-map}
\end{figure}

\section{When is $\sigma(a)$ on the lower bound $\sigma_1(a)$?}

The preceding results (and particular the graph of Figure~\ref{fig:bounds}) lead us naturally to seek some criterion for when $\sigma(a) = \sigma_1(a)$.  We prove the following result:

\begin{thm}
Let $a \in \Z^+$.  Then $\sigma_1(a)=\sigma(a)$ if and only if one of the following conditions holds:
\begin{enumerate}

\item  $\lfloor \sigma_l(a) \rfloor < \lfloor \sigma_l (a-1) \rfloor$
\item  $\lfloor \sigma_r(a) \rfloor <\lfloor \sigma_r(a+1) \rfloor$

\end{enumerate}
\end{thm}

To prove this theorem, we may consider two distinct cases:  (i)  $n^2 < a \leq n^2 + n$, and (ii) $n^2 + n + 1 \leq a < (n+1)^2$.  Recall that in the first case, we write $a=n^2+b$ and have $\sigma_1(a)=\sigma_l(a)$; we must prove that in this case $\sigma_l(a)=\sigma(a) \iff \sigma_l(a)<\sigma_l(a-1)$.  For the second case, we write $a=(n+1)^2-c$, and have $\sigma_1(a)=\sigma_r(a)$; we must then prove that $\sigma_r(a)=\sigma(a) \iff \sigma_r(a)<\sigma_r(a+1)$.   We present below the proof of the first case; the second case follows similar lines.

Let $k \in \mathbb Z$ be arbitrary, and for $i \in \{0,1\}$ define

\begin{equation}
\delta_i:=k- \frac{n+\sqrt{n^2+b+i}}{b+i}
\end{equation}

Then the hypothesis that $\sigma_l(n^2+b)<\sigma_l(n^2+b-1)$ is equivalent to the existence of $k$ lying strictly between $\frac{n+\sqrt{n^2+b+1}}{b+1}$ and $\frac{n+\sqrt{n^2+b}}{b}$, which is in turn equivalent to the existence of $k$ such that $\delta_0<0<\delta_1$. 

Next we note that for $i\in\{0,1\}$,
\begin{equation}
\begin{aligned}
k \sqrt{n^2+b+i} &= \left( \frac{n+\sqrt{n^2+b+i}}{b+i}+\delta_i\right)\sqrt{n^2+b+i}\\
&=\frac{n^2+b+i+n\sqrt{n^2+b+i}}{b+i}+\delta_i\sqrt{n^2+b+i}\\
&=1+n\cdot \frac{n+\sqrt{n^2+b+i}}{b+i}+\delta_i\sqrt{n^2+b+i}\\
&=1+nk+\delta_i(\sqrt{n^2+b+i}-n).
\end{aligned}
\end{equation}

Note that
$$n < \sqrt{n^2+b+i} < n+1  $$
so that
\begin{equation}  0 < \sqrt{n^2+b+i}-n < 1 \end{equation}

Combining (2) and (3), we find that $k$ lies strictly between $\frac{n+\sqrt{n^2+b+1}}{b+1}$ and $\frac{n+\sqrt{n^2+b}}{b}$ if and only if $1+nk$ lies strictly between $k\sqrt{n^2+b}$ and $k\sqrt{n^2+b+1}$.  We are now ready to prove the following proposition, from which the theorem follows.

\begin{prop} There exists an integer strictly between $\frac{n+\sqrt{n^2+b+1}}{b+1}$ and $\frac{n+\sqrt{n^2+b}}{b}$ 
if and only if there exists an integer strictly between $\sigma_l(n^2+b)\sqrt{n^2+b}$ and $\sigma_l(n^2+b)\sqrt{n^2+b+1}$.
\end{prop}

\begin{proof}
Let $k=\sigma_l(n^2+b)$ .  Assume  there exists an integer strictly between $\frac{n+\sqrt{n^2+b+1}}{b+1}$ and $\frac{n+\sqrt{n^2+b}}{b}$.
Then $k$ is certainly such an integer, because by definition $\sigma_l(n^2+b)$ is the smallest integer strictly larger than $\frac{n+\sqrt{n^2+b+1}}{b+1}$.
Then by the lemma, $1+nk$ is  strictly between $k\sqrt{n^2+b}$ and $k\sqrt{n^2+b+1}$.

Conversely assume that $m$ is an integer strictly between $k\sqrt{n^2+b}$ and $k\sqrt{n^2+b+1}$. 
Define $\delta_{0,1}$ per $(1)$.  By the definition of $k$, we automatically have $0<\delta_1\le 1$.
Then from $(2)$ and $(3)$
$$m<k\sqrt{n^2+b+1}=1+nk+\underbrace{\delta_1(\sqrt{n^2+b+1}-n)}_{\in(0,1)}$$
so that $m\le 1+nk$. Using $(3)$ again we get
$$ 1+nk\ge m>k\sqrt{n^2+b}=1+nk+\delta_0(\sqrt{n^2+b}- n)$$
so that $\delta_0(\sqrt{n^2+b}- n)<0$, and by $(11)$ we conclude $\delta_0<0$. But $\delta_0<0<\delta_1$ is equivalent to $k$ being strictly between 
 $\frac{n+\sqrt{n^2+b+1}}{b+1}$ and $\frac{n+\sqrt{n^2+b}}{b}$.

\end{proof}

\section{Points above $\sigma_1$}

To this point we have done much to explain the data:  We have found both upper and lower bounds, and established criteria for precisely when those bounds are sharp.  As noted previously, that in the interval $0 \leq a \leq 500$,  62.8\% of the data points lie on the lower-bound curve $\sigma_1(a)$ given by Theorem 2.  What about the rest?

Towards understanding the values of $\sigma(a)$ for which $\sigma(a)>\sigma_1(a)$, we we note that $\tau_s(a) = 0 \iff$ for some unique $k\geq 0$,
\begin{equation}
(sn+k)^2 \leq s^2a \mbox{~and~} (sn+k+1)^2 \geq s^2(a+1)
\end{equation}

Solving (4) for $s$, we find that it is satisfied provided 
\begin{equation}
k \frac{n+\sqrt{a}}{b} \leq s \leq (k+1) \frac{n + \sqrt{a+1}}{b+1}
\end{equation}

Similarly, we also observe that $\tau_s(a) = 0 \iff$ for some unique $k\geq 0$,
\begin{equation}
(sm-k)^2 \geq s^2(a+1) \mbox{~and~} (sm-k-1)^2 \leq s^2(a)
\end{equation}

Solving (6) for $s$, we find that it is satisfied provided 
\begin{equation}
k \frac{m+\sqrt{a+1}}{c-1} \leq s \leq (k+1) \frac{m + \sqrt{a}}{c}
\end{equation}

Combining the upper bounds of both (6) and (7), we recognize again the significance of the functions $\sigma_k$ given by
$$\sigma_k(a) = \max \{ \lfloor (k)\frac {n + \sqrt{a+1}}{b+1} \rfloor +1 , \lfloor (k)\frac {m + \sqrt{a}}{c} \rfloor +1 \} $$

Figure~\ref{fig:sigma-k} shows computed values of $\sigma(a)$ together with curves corresponding to $\sigma_k$   for $1 \leq k \leq 15$ and $k=18,19,22,29,\mbox{~and~}40$ on the interval $100^2 \leq a \leq 101^2$.  Note that each of the data points $(a, \sigma(a))$ lies on one of the curves $\sigma_k$.  This representation of the data, useful as it is, nevertheless conceals some other patterns which may be found in the values of $\sigma(a)$.  We turn to these other patterns in the next section.

\begin{figure}[htbp]
  \centering
  \includegraphics[scale=.40]{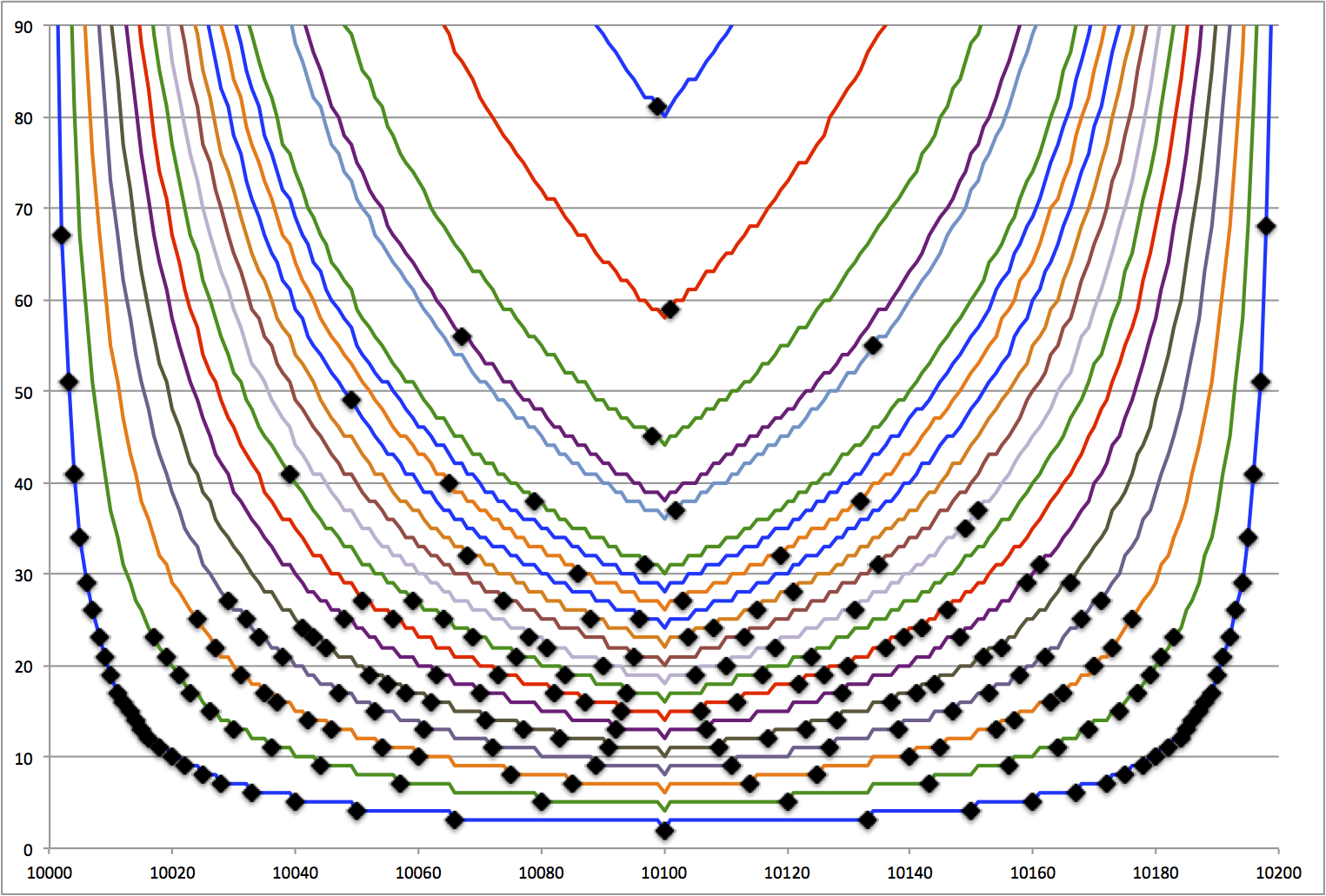}
  \caption{$\sigma(a)$ for $100^2 \leq a \leq 101^2 $ with $\sigma_k(a)$ for $1 \leq k \leq 15$ and $k=18,19,22,29,\mbox{~and~}40$}
  \label{fig:sigma-k}
\end{figure}

\section{Where to next?  Some observations and questions that seem important}

In this section I describe some features of the sequence $(\tau_s(a))_{s \in \Z^+}$ and the function $\sigma(a)$ that have not yet been accounted for.  The observations in this section are meant to be descriptive and suggestive only; where the data suggests possible future directions, they are listed as questions and/or conjectures.

\subsection{Near-symmetry in the graph of $\sigma(a)$}

If we refer to Figure~\ref{fig:sigma-k}, we may observe that many, \textit{but not all}, of the points $(a, \sigma(a))$ come in symmetric pairs.  More precisely, letting $a_{min} = n(n+1)$ be the value for which $\sigma(a_{min})=2$, then $\sigma(a_{min}-d) = \sigma(a_{min}+d)$ holds for approximately 60\% of 
of all $d$ with $0 \leq |d| \leq a_{min}$.  As yet I have no criterion for determining the values of $d$ for which   $\sigma(a_{min}-d) \neq \sigma(a_{min}+d)$.

\subsection{Structure of the sequence $\tau_s(a)$}

A straightforward computation shows that for $a=n^2$ or $a=n^2-1$, if 
$$ s \sqrt{a} < t < s \sqrt{a +1} $$
then
$$ (s+1) \sqrt{a} < t+n < (s+1) \sqrt{a^2 +1} $$
From this observation it follows that the sequence $\tau_s(a)$ is non-decreasing for such values of $a$.    

Figure~\ref{fig:T-sets-heat-map}  and Figure~\ref{fig:delta-tau-heat-map} suggest that these are the only cases in which this happens.  More specifically, we propose:

\begin{conj}
If neither $a$ nor $a+1$ is a square integer, then for any $k>0$ there exists some $s$ such that $\tau_s(a)=k$ but $\tau_{s+1}(a)=k-1$.
\end{conj}
Note that the proof of Thm 1(2) shows that, if this conjecture is true, it must be due to the existence of some $t$ such that $s \sqrt{a} < t < s \sqrt{a+1}$ but $t + n < (s+1)\sqrt{a}$. As noted above, such a $t$ does not exist if $a=n^2$ or $a=n^2-1$.

The above observation can also be expressed as follows:   Define the set $S(a) = \{ s: \tau_s(a) > 0 \}$.  Then if $a=n^2$ or $a=n^2-1$ , the set $S(a)$ is upward-closed, i.e. $s \in S(a) \implies s+1 \in S(a)$.  However, $S(a)$ is not upward-closed in general.  For example, we have already shown that $2 \in S(n^2 + n)$ for all $n$, but it may easily be verified that $3 \notin S(n^2 + n)$ for any $n>1$.  In fact $S(n^2 + n)$ contains all even positive integers, but for any odd positive integer $k$ there exists $n_0$ such that for $n > n_0$, $k \notin S(n^2 + n)$.  

\subsection{For given $a$, how can we determine $k$ such that $\sigma(a) = \sigma_k(a)$?}  Figure~\ref{fig:sigma-k} and the discussion at the end of the previous section shows that for each $a$, there exists $k$ such that $\sigma(a) = \sigma_k(a)$.  However, there does not seem to be any obvious relationship between $a$ and $k$; we note that within the interval $100^2 < a < 101^2$ the set of values of $k$ for which $\sigma(a) = \sigma_k(a)$ for some $a$ is precisely $ \{1, 2, 3, \dots, 15, 18, 19, 22, 29, 40 \}$.  More generally, for any $n$ if we look on the interval $n^2 < a < (n+1)^2$, for what values of $k$ will there be some $a$ such that $\sigma(a) = \sigma_k(a)$?

\subsection{Is there a fractal structure in the graph of $\sigma(a)$?} 

In the first section of these notes I referred to the graph of $\sigma(a)$ as ``seemingly chaotic''.  Is there a sense in which this can be made precise?

Towards an answer to this question, we consider Fig.~\ref{fig:off-bound}, which graphs $\sigma(a)$ on the interval $0 \leq a \leq 2000$  \textit{only} for those values of $a$ for which $\sigma(a) > \sigma_1(a)$.  (These will be referred to below as ``off-bound points''.)  We note that Fig.~\ref{fig:off-bound} contains certain features that are reminiscent of Fig.~\ref{fig:sigma}.  In particular:

\begin{figure}[htbp]
  \centering
  \includegraphics[scale=.5]{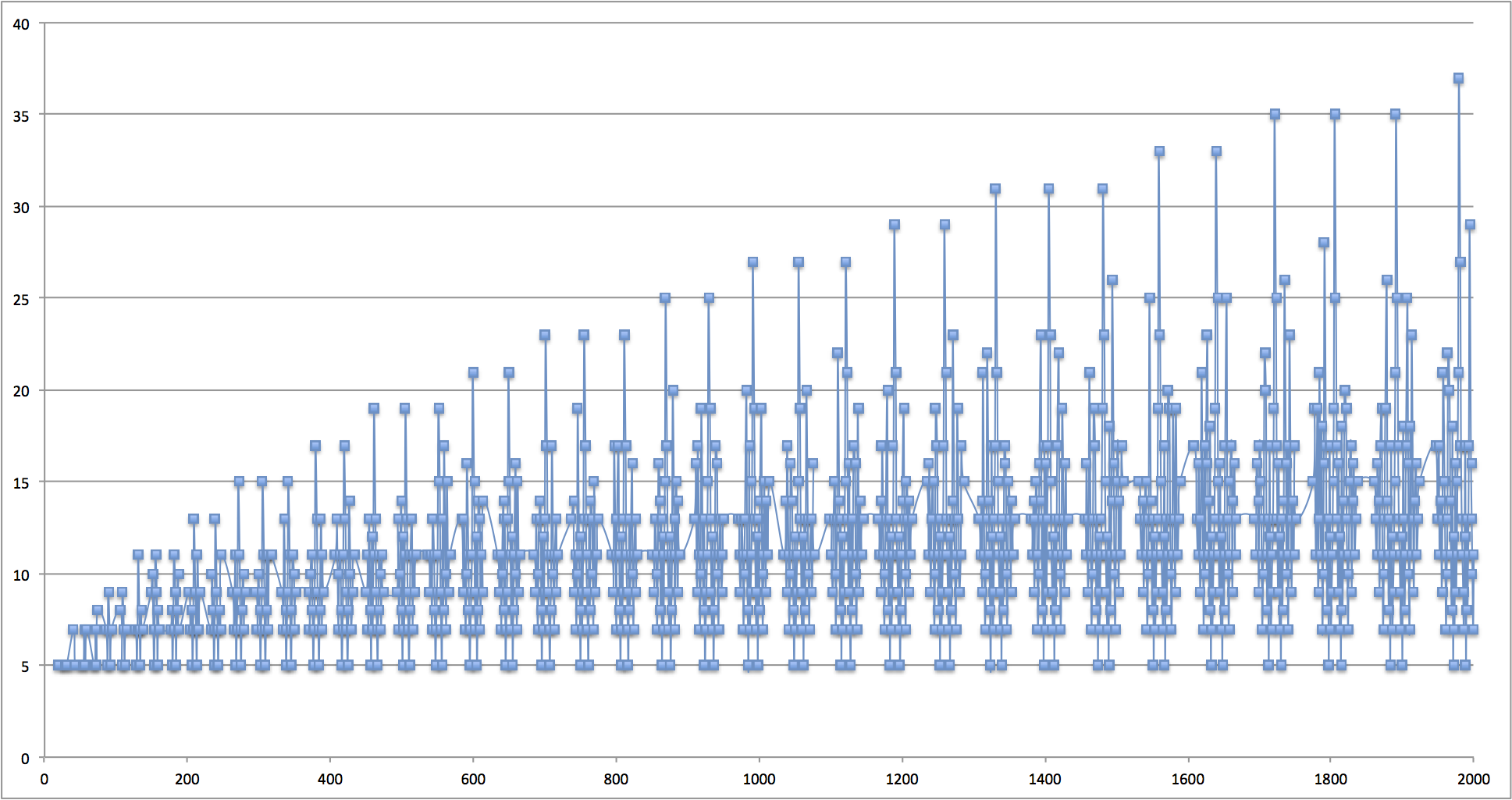}
  \caption{Off-bound points for $\sigma(a)$, $0 \leq a \leq 2000$}
  \label{fig:off-bound}
\end{figure}

\begin{itemize}

\item Off-bound values of $\sigma(a)$ fluctuate with intermittent spikes, located at values of $a$ that are of the form $a=n^2 + n - 1$; these peak values are located almost exactly halfway between the peak values of the full graph, which are located at $a=n^2$.

\item The values of $\sigma(a)$ at these off-bound peak values form the sequence of odd numbers, in alternating groups of $3$ and $2$; for example:
\begin{itemize}

\item The three off-bound peaks between $a=11^2$ and $a=14^2$ all have $\sigma(a)=11$;
\item The two off-bound peaks between $a=14^2$ and $a=16^2$ all have $\sigma(a)=13$;
\item The three off-bound peaks between $a=16^2$ and $a=19^2$ all have $\sigma(a)=15$;
\item The two off-bound peaks between $a=19^2$ and $a=21^2$ all have $\sigma(a)=17$; and so on.
\end {itemize} 

\item For $n \geq 7$, within each interval $n^2 \leq a \leq (n+1)^2$ there are precisely two off-bound points that attain the minimum value $\sigma(a)=5$; one of these minima occurs between $n^2$ and $n^2 + n - 1$, and the other occurs between $n^2 + n - 1$ and $(n+1)^2$.

\end{itemize}

\begin{figure}[htbp]
  \centering
  \includegraphics[scale=.5]{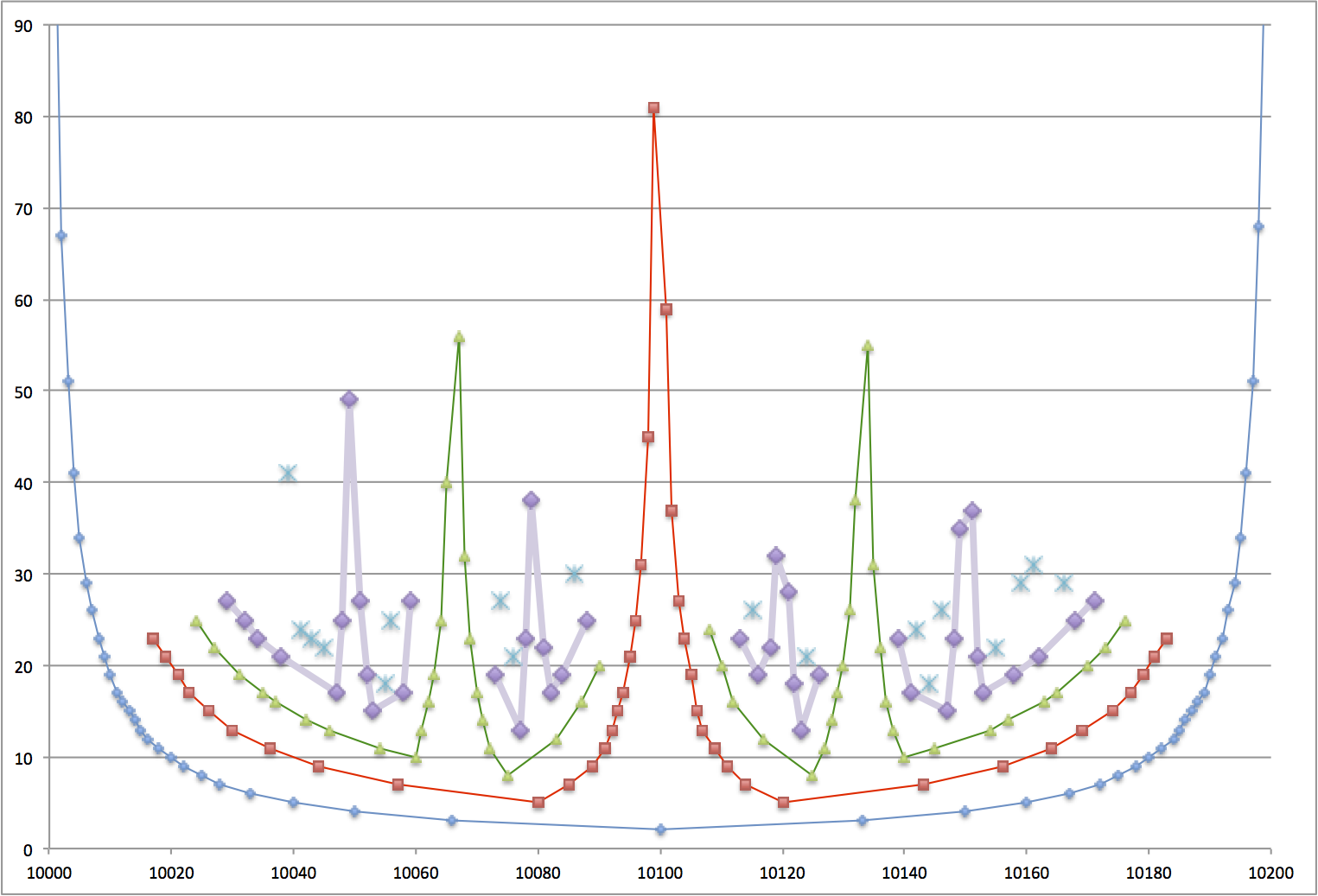}
  \caption{Apparent recursive structure in the on- and off-bound points for $\sigma(a)$, $100^2 \leq a \leq 101^2$.}
  \label{fig:recursive}
\end{figure}

We can summarize these observations rather loosely by saying that within each of the cup-shaped troughs of $\sigma_1$ between consecutive square integers, the off-bound points form \textit{two} cup-shaped troughs with roughly the same characteristics as the ``main'' trough.  Figure~\ref{fig:recursive} suggests that, as $n$ increases and the fraction of points between $n^2$ and $(n+1)^2$ that are off-bound grows, this doubling phenomenon may repeat recursively: note that the region above the lower bounding curve $\sigma_1(a)$ is subdivided into two subregions, each of which is in turn subdivided into two sub-subregions, and so on.  Although the later curves in this recursive process are more irregular then the earlier curves, in general as $a$ gets larger these curves become increasingly smooth and symmetric.

Doubling phenomena such as these are often characteristic of chaotic dynamical systems, and strongly suggest that there may be some way of interpreting these phenomena in terms of some underlying (but as yet unknown) dynamical process, with respect to which the functions may be understood as genuinely chaotic.

\section*{Acknowledgement}
Early versions of this paper were improved substantially by the generous comments of Hyman Bass, Benjamin Schmidt, and anonymous reviewers.  I am also indebted to Hagen von Eitzen, who provided most of the proof of Theorem 3 in a discussion on the \href{http://math.stackexchange.com/}{Mathematics Stack Exchange} website at \url{http://math.stackexchange.com/questions/1238197/}.

%\subsection{}

\end{document}